\documentclass{article}
\usepackage{latexsym}
\usepackage{eepic}
\usepackage{makeidx}
\usepackage[dvips]{graphicx}
\usepackage[latin1]{inputenc}
\usepackage[english]{babel}
\usepackage{fancybox} 
\usepackage{amstext,amssymb,amsmath,amsthm}
\begin{document} 
\newtheorem{theorem}{Theorem}[section]
\newtheorem{lemma}{Lemma}[section]
\newtheorem{conjecture}{Conjecture}[section]
\newtheorem{definition}{Definition}[section]
\newtheorem{remark}{Remark}[section]
\newtheorem{claim}{Claim}[section]
\newcommand{\oh}{\left(\frac{1}{2}\right)}
\newcommand{\of}{\left(\frac{1}{4}\right)}
\newcommand{\nt}{\not\to}
\newcommand{\lra}{\leftrightarrow}
\newcommand{\p}{^\prime}
\newcommand{\mbf}{\mathbf}
\newcommand{\ol}{\overline}
\newcommand{\ora}{\overrightarrow}

\title{Path correlations in a randomly oriented complete bipartite graph}
\author{Erik Aas}
\date{}

\abstract{In a randomly oriented graph containing vertices $x$ and $y$, denote by $\{x\to y\}$ the event that there is a directed path from $x$ to $y$. We study the correlation between the events $\{x\to y\}$ and $\{y\to z\}$ for a (large) oriented complete bipartite graph with orientation chosen uniformly at random. We classify the cases of positive and negative correlation respectively in terms of the relative proportions of the sizes of the color classes of the graph.}

\maketitle
\section{Introduction}

Let $G$ be an aribitrary finite graph whose edges $e$ all have been assigned probabilities $p(e)$. We get a random graph by including the edge $e$ in the graph with probability $p(e)$ independently of everything else. Let $\{x \lra y\}$ denote the event that there is a path connecting the two nodse $x$ and $y$ in this random graph.

In \cite{mcdiarmid} it was observed that when all $p(e)=\frac{1}{2}$, the probability $P(x\lra y)$ coincided with the probability of $\{x\to y\}$, the event of there being a directed path from $x$ to $y$ in a uniformly chosen random orientation of the edges of $G$ (so here we are comparing the probabilities of two different events in two different probability spaces). Clearly choosing an orientation uniformly at random amounts to orient each edge either way with equal probability $1/2$, independently of all other edges.

A natural question, then, is what other properties those two probability spaces share.

It is a basic fact \cite{harris} that in the undirected random graph model defined above, for any increasing events $A$ and $B$ (like $\{x\lra z\}$ and $\{z\lra y\}$),
$$P(A\cap B) \geq P(A)P(B),$$
that is, $A$ and $B$ are positively (or, to be pedantic, nonnegatively) correlated.

An analogue of this fact, stated in \cite{mcdiarmid}, is that in any {\it randomly directed} graph, the events $\{x\to y\}$ and $\{y\to z\}$ are positively correlated, that is,
$$
P(x\to y \to z) \geq P(x\to y)P(x\to z).
$$

This motivates the following question:
{\it Let $G$ be a graph containing three vertices $x$,$y$,$z$. Does 
$$P(x \to y \to z) \geq P(x\to y)P(y\to z)$$
hold, that is, are $\{x\to y\}$ and $\{y \to z\}$ positively correlated?} Here, and in the following, we write $\{x\to y\to z\} := \{x\to y\}\cap\{y\to z\}$.

Obviously this depends on the graph $G$, as it is easy to find graphs for which $P(x\to y\to z)-P(x\to y)P(y\to z)$ has any given sign (including $0$), and a simple characterisation of all graphs with, say, positive correlation seems hard to find.
It is known \cite{seasl09} that for the complete graph $K_n$, this quantity is negative for $n=3$, zero for $n=4$ and positive for all $n>4$. For a slightly different model, in \cite{gnp} it is shown that when the graph is 'dense', the analogous correlation is positive.
Here we will study the same correlation between the events $\{x\to y\}$ and $\{y\to z\}$ in a uniformly chosen orientation of the edges of the complete bipartite graph $K_{m,n}$.

\section{Result}

Throughout the remainder of this note, let $A = \{x\to y\}$ and $B = \{y\to z\}$. We denote the complement of a set (or an event) $A$ by $A^c$. We define $\{x\nt y\}$ to be $\{x\to y\}^c$, that is, there is not path from $x$ to $y$.
As a technical convenience the object of study will be $RC_{m,n} := \frac{P(A^c\cap B^c) - P(A^c)P(B^c)}{P(A^c\cap B^c)}$, {\it the relative covariance between $A^c$ and $B^c$}, rather than $P(A\cap B)-P(A)P(B)$.
Observe that $P(A\cap B)-P(A)P(B) = P(A^c \cap B^c)-P(A^c)P(B^c)$ - this holds for any two events $A$, $B$.
In particular, $A$ and $B$ are positively correlated if and only if $RC_{m,n}$ is positive.
Observe that the relative covariance can be rewritten as (and this might be the more convenient way of thinking about it) $RC_{m,n} = 1 - \frac{P(A^c)P(B^c)}{P(A^c \cap B^c)} = 1 - \frac{P(B^c)}{P(B^c | A^c)}$.

We fix some more notation:
\begin{itemize}
\item The nodes of $K_{m,n}$ are partitioned into two sets $X$ and $Y$ of sizes $m$ and $n$ respectively.
\item $m = \lfloor \beta n \rfloor$ for some fixed positive constant $\beta$.
\item $m,n\geq 2$.
\item $a,b,c,d$ are four distinct vertices of $K_{m,n}$; the first three belong to $X$ and $d$ belongs to $Y$.
\item The limit $\lim_{n\to\infty} RC_{m,n}$ is denoted by $RC$.
\end{itemize}

\begin{theorem}
\label{main}
The value of $RC$ is given by the following table.

\begin{table}[hb]
\begin{tabular}{r|r||r|r|r}
$X$& $Y$& $\beta<1$& $\beta=1$& $\beta>1$\\
\hline
&\\[-10pt]
$x,y,z$& & -1/3&-1/3&-1/3\\
\hline
$x,y$&$z$& 1/2& 1/5&-1\\
\hline
$x,z$&$y$&  1& 1/5& 0\\
\hline
\end{tabular}

\caption{The relative covariance between $\{x\nt y\}$ and $\{y\nt z\}$, according to which partition the vertices belong to, and to the proportion $\beta$ of the number $X$-vertices to the number of $Y$-vertices.}
\end{table}
\end{theorem}

We see that letting $m = \lfloor \beta n\rfloor$ for a fixed constant $\beta$ is not as restrictive as might seem at first thought.

\section{Proof}

The proof of Theorem \ref{main} will follow from a number of lemmas estimating the probabilities $P(A^c), P(B^c), P(A^c\cap B^c)$ in terms of $n$.
A common feature of these estimates is that the lower bounds, which are trivial to obtain, are close to the harder-to-prove upper bounds.
If $S$ and $T$ are two disjoint sets of vertices in $K_{m,n}$, an '$ST$-witness' is defined to be a vertex $u$ for which there is at least one edge from $S$ to $u$ and at least one edge from $u$ to $T$.

For a vertex $s$ in $K_{m,n}$, the set $O_s$ will loosely be defined as the set of vertices in $X^\prime\cup Y^\prime$ (thus in $X^\prime$ if $a\in Y$ and in $Y^\prime$ if $a\in X$) which can be reached in exactly one step from $s$, $X^\prime$ and $Y^\prime$ being defined separately in each section where this notation is used.
We denote $|X^\prime|$ and $|Y^\prime|$ by $m^\prime$ and $n^\prime$.
The set $I_a$ of vertices which reach $a$ in exactly one step is similarly defined.

Estimating these probabilities will be a lot of repetitive work.
The following inequality will be used several times: if $s$, $t \geq \alpha$, then $st \geq \alpha s + \alpha t - \alpha^2$.

Below, when summing over subsets of nodes denoted by upper case letters, the sizes of these sets will often be denoted by the corresponding lower case letters.

To estimate sums of the form $\sum _{s = 0} ^n \sum_{t=0} ^{n-s} {n\choose s} {n-s \choose t} \oh ^{st}$ we will split them into several parts according to whether $s\geq \alpha$ or $t \geq \alpha$ for some suitably chosen constant $\alpha$ (depending only on $\beta$).

\subsection*{(i) $P(b\not\to a)$}

\begin{lemma}
\label{ba}
$$
P(b\nt a) \sim 2\oh^n.
$$
\end{lemma}
\begin{proof}
Let $X^\prime = X-\{a,b\}$, $Y^\prime=Y$.

A lower bound is given by
$P(b\not\to a) \geq 
P(\{\text{there is no edge directed away from } b\} \cup \{\text{there is no edge directed towards}a\}) = 
2\oh^{n}-\oh^{2n}$, by inclusion-exclusion.
By calculating the probability that there is no path from $b$ to $a$ of length at most $4$, we get the following upper bound: $P(b\not\to a) = \sum_{S,T\subseteq Y} P(b\nt a|O_b = S, I_a = T)P(O_b=S,I_a=T) \leq
\oh^{2n}\sum_{S,T\subseteq Y: S\cap T=\emptyset} P(\text{no } x\in X^\prime\text{ is an } ST\text{-witness}) =
\oh^{2n}\sum_{s=0}^n\sum_{t=0}^{n-s}{n\choose s}{n-s\choose t}\left(\oh^s+\oh^t-\oh^{s+t}\right)^{m-2}$.

Note that the partial sum corresponding to $st=0$ is equal to the lower bound.
We now show that the other terms sum to $o(\oh^{n})$.
Split the remaining sum into the following four parts: 
$S_1$: $s$, $t  \geq \alpha$;
$S_2$: $1\leq s \leq \alpha \leq t$;
$S_3$: $1\leq t \leq \alpha \leq s$;
$S_4:1\leq s,t\leq \alpha$.

Note that in the $S_1$ case, $\oh^s+\oh^t-\oh^{s+t}\leq \oh^{\alpha-1}$.
Hence $S_1 \leq \oh^{2n}\sum_{s=\alpha}^n{n\choose s}\sum_{t=\alpha}^{n-s}{n-s\choose t} \oh^{(\alpha-1)(m-2)} =
\oh^{2n+(\alpha-1)(m-2)}3^n = o(\oh^{(\alpha-1)(m-2)})=o(\oh^{n})$, the last equality holding when choosing $\alpha$ large enough.

$S_2 \leq 
\oh^{2n}\sum_{s=1}^\alpha{n\choose s}\sum_{t=\alpha}^{n-s}{n-s\choose t}\left(\oh^s+\oh^t\right)^{m-2}\leq
\oh^{2n}\sum_{s=1}^\alpha n^\alpha \sum_{t=0}^n {n\choose t}\left(\oh+\oh^\alpha\right)^{m-2}\leq
\oh^n\alpha n^\alpha \left(\oh+\oh^\alpha\right)^{m-2}=o(\oh^n)$, if $\alpha>0$.

By symmetry, we may choose $\alpha$ possbily even larger so that $S_3 = o(\oh^{n})$ holds.

Clearly, $S_4 = o(\oh^n)$.

Hence $P(b\nt a) - 2\oh^n \leq S_1+S_2+S_3+S_4 = o(\oh^{n})$.
\qedhere
\end{proof}

\subsection*{(ii) $P(d\nt a)$}
\begin{lemma}
\label{da}
$P(d\nt a) \sim \oh^{m} + \oh^{n}$. (This is $\sim \oh^{m}$ for $\beta < 1$, $\sim 2\oh^{n}$ for $\beta = 1$, and $\sim \oh^n$ for $\beta > 1$.)
\end{lemma}
\begin{proof}
Let $X^\prime = X -\{a\}$, $Y^\prime =Y -\{d\}$.
The probability is bounded from below by $P(d\nt a) \geq P(\text{no edge leaves } d \text{ or no edge enters } a) \geq \oh^m+\oh^n-\oh^{m+n-1}$.

For the upper bound, we calculate the probability that there is no path from $d$ to $a$ of length at most $3$:
$P(d\nt a) = 
\sum_{S\subseteq Y^\prime, T\subseteq X^\prime}P(a\nt d|I_a =S,O_d=T)P(I_a=S,O_d=T)\leq
\oh^{m+n-1}\sum_{s=0}^{n-1}{n-1\choose s}\sum_{t=0}^{m-1}{m-1\choose t}P(\text{no edge from } S \text{ to } T) =
\oh^{m+n-1}\sum_{s=0}^{n-1}{n-1\choose s}\sum_{t=0}^{m-1}{m-1\choose t}\oh^{st}$.
The partial sum with $st=0$ equals the lower bound.
We now show that the remaining terms sum to $o(\oh^m+\oh^n)$ by splitting their sum into the following cases:
$S_1$: $s$, $t\geq \alpha$, 
$S_2$: $1\leq s\leq \alpha$, and
$S_3$: $1\leq t\leq \alpha$.

Using $s$, $t \geq \alpha \Rightarrow st \geq \alpha s + \alpha t - \alpha^2$,
$S_1 =
\oh^{m+n-1}\sum_{s=\alpha}^{n-1}{n-1\choose s}\sum_{t=\alpha}^{m-1}{m-1\choose t}\oh^{st}\leq
\oh^{m+n-1}\sum_{s=\alpha}^{n-1}{n-1\choose s}\sum_{t=\alpha}^{m-1}{m-1\choose t}\oh^{\alpha s+\alpha t-\alpha^2}\leq
\oh^{m+n-1}2^{\alpha^2}\sum_{s=0}^{n-1}{n-1\choose s}\oh^{\alpha s} \sum_{t=0}^{m-1}{m-1\choose t}\oh^{\alpha t}\leq
\oh^{m+n-1}2^{\alpha^2}\left(1+\oh^\alpha\right)^{m+n-2} = o(\oh^m+\oh^n)$, choosing $\alpha$ large enough. 
Similarly,
$S_2= \oh^{m+n-1}\sum_{s=1}^\alpha {n-1\choose s}\sum_{t=0}^{m-1}{m-1\choose t}\oh^{st} \leq
\oh^{m+n-1}(n-1)^\alpha\alpha \sum_{t=0}^{m-1} {m-1\choose t}\oh^t =
\alpha(n-1)^\alpha\oh^n\left(\frac{3}{4}\right)^{m-1} = o(\oh^m+\oh^n)$.

A similar argument shows $S_3 = o(\oh^m+\oh^n)$.

Hence $P(d\nt a) - (\oh^m+\oh^n) \leq S_1+S_2+S_3 = o(\oh^m+\oh^n)$.
\qedhere
\end{proof}

\subsection*{(iii) $P(b\nt d\nt a)$}

\begin{lemma}
\label{bda}
$P(b\nt d\nt a) \sim 2\oh^{m+n-1}+\oh^{2n}$
\end{lemma}

\begin{proof}
Let $X^\prime = X-\{a,b\}$, $Y^\prime=Y-\{d\}$.
For the lower bound, we calculate the probability $P(\{$ the edge between $b$ and $d$, and the edge between $d$ and $a$, form a directed path from $a$ to $b\}\cap (\{O_b = O_d=\emptyset\}\cup\{O_b=I_a=\emptyset\}\cup\{I_d=I_a=\emptyset\}))$, which is $2\oh^{m+n-1}+\oh^{2n}-\oh^{m+2n-3}$, by inclusion-exclusion;
$P(b\nt d\nt a) \geq 2\oh^{m+n-1}+\oh^{2n}-\oh^{m+2n-3}$.

To get a working upper bound, it is sufficient to calculate the probability of there being no path from $b$ to $d$ or from $d$ to $a$, either of length at most $3$.
Conditioning on $I_a=S,O_b=T,I_d=U,O_d=V$, there may be no edge from $T$ to $U$, nor from $V$ to $S$.
The edges $\{b,d\}$ and $\{a,d\}$ form a directed path from $a$ to $b$.
This implies that $S$ and $T$ must be disjoint.

Hence
$$P(b\nt d\nt a)$$

$$=\sum_{S,T,U,V}P(b\nt d\nt a|I_a=S,O_b=T,I_d=U,O_d=V)P(I_a=S,O_b=T,I_d=U,O_d=V)$$

$$\leq \oh^{m+2n-2}\sum_{s=0}^{n-1}{n-1\choose s}\sum_{t=0}^{n-1-s}{n-1-s\choose t}\sum_{u=0}^{m-2}{m-2\choose u}\oh^{tu+s(m-2-u)}.$$

The sum of the terms for which $s=t=0$, $t=u=0$ or $t=u-(m-2)=0$ equals the lower bound.
The remaining sum is split into the following cases:
$S_1$: $s$, $t\geq \alpha$, $S_2$: $1\leq t\leq \alpha\leq s$, $S_3: 1\leq s\leq \alpha\leq t$, and $S_4:1\leq s,t\leq \alpha$.

$S_1\leq
2^{2\alpha^2}
\oh^{2n+m-2}\sum_{s=\alpha} ^{n-1} {n-1\choose s} \sum _{t=\alpha}^{n-1-s}{n-1-s\choose t}\sum_{u=0}^{m-2}\oh^{\alpha(m-2)}=\\
\oh^{2n+m-2}3^{n-1}\oh^{(\alpha-1)(m-2)}=
o(\oh^{m+n-2}+\oh^{2n})$.

$S_2 \leq
\oh^{2n+m-2}\sum_{s=1}^{\alpha}{n-1\choose s}\sum_{t=\alpha} ^{n-1-s} {n-1-s\choose t} \sum_{u=0}^{m-2}\oh^{tu+s(m-2-u)}\leq
\oh^{2n+m-2}(n-1)^\alpha\alpha \sum_{t=\alpha}^{n-1}{n-1\choose t}\sum_{u=0}^{m-2}{m-2\choose u}\oh^{\alpha u+m-2-u}=
\oh^{2(m-2+n)}(n-1)^\alpha\alpha2^{n-1}\left(1+\oh^{\alpha-1}\right)^{m-2} = o(\oh^{m+n-1}+\oh^{2n})$.

By symmetry with $S_2$, we deduce $S_3 = o(\oh^{m+n-1}+\oh^{2n})$.

$S_4 \leq \oh^{2n+m-2}(n-1)^{2\alpha}\alpha^2 = o(\oh^{m+n-1}+\oh^{2n})$.

We conclude that 
$$
P(b\nt d\nt a) \sim 2\oh^{m+n-1}+\oh^{2n}.
$$
\end{proof}

\subsection*{(iv) $c\nt b\nt a$}

\begin{lemma}
\label{cba}
$$
P(c\nt b\nt a) \sim 3\oh^{2n}
$$
\end{lemma}
\begin{proof}

For the lower bound, note that $P(c\nt b\nt a) \geq P(\{O_c = O_b = \emptyset\}\cup\{O_c = I_a = \emptyset\}\cup\{I_b = I_a= \emptyset\}) \geq 3\oh^{2n} -2\oh^{3n}$.

For the upper bound, we will sum over $U\subseteq Y^\prime$, $V=Y^\prime-U$, $S\subseteq U$, and $T\subseteq V$.
When doing so, an expression for the probability that a given vertex $x\p \in X\p$ is not a $TU$-witness and not a $VS$-witness is needed.
The probability of the complementary event is the probability of $x\p$ being a $TU$-witness or a $VS$-witness.
The separate probabilities for these last two events are $P(x\p$ is a $TU$-witness$) = \left(1-\oh^{|T|}\right)\left(1-\oh^{|U|}\right)$ and $P(x\p$ is an $SV$-witness$) = \left(1-\oh^{|S|}\right)\left(1-\oh^{|V|}\right)$. The probability of their intersection is $\left(1-\oh^{|S|}\right)\left(1-\oh^{|T|}\right)$, using $S\subseteq U$, $T\subseteq V$, 
By inclusion-exclusion, we get 
$P(x\in X^\prime \text{ is not a } TU \text{-witness, nor a } VS\text{-witness})= 
1-\left(1-\oh^{|T|}\right)\left(1-\oh^{|U|}\right)-\left(1-\oh^{|S|}\right)\left(1-\oh^{|V|}\right)+  \left(1-\oh^{|S|}\right)\left(1-\oh^{|T|}\right)$, which simplifies to
$\oh^{|U|}+\oh^{|V|}-\oh^{|T|+|U|}-\oh^{|S|+|V|}+\oh^{|S|+|T|}$. 

$$P(c\nt b\nt a) $$

$$=\sum_{S,T,U,V}P(c\nt b\nt a|I_a = S, O_c = T, I_b = U, O_b = V)P(I_a = S,O_c = T,I_b = U,O_b = V)$$

$$\leq\sum_{S,T,U,V} P(\text{no } x\in X^\prime \text{ is a } TU \text{-witness, nor a } VS\text{-witness})\oh^{3n}$$

$$=\oh^{3n}\sum_{u=0, u+v=n}^{n}{n\choose u}\sum_{s=0}^u{u\choose s}\sum_{t=0}^v{v\choose t}\left(\oh^u+\oh^v-\oh^{t+u}-\oh^{s+v}+\oh^{s+t}\right)^{m-3}.$$

The sum of the terms with $s = t = 0$ or $u=0$ or $v=0$ equals the lower bound.
The other terms sum to $o(\oh^{2n})$, as we now turn to show.
Since $(u,t)$ and $(v,s)$ are interchangeable, we need only consider the following cases:
$S_1$: $s$, $t\geq \alpha$;
$S_2$: $1\leq s\leq\alpha \leq t,u$;
$S_3$: $1\leq s\leq\alpha\leq t$, $1\leq u\leq \alpha$;
$S_4$: $1\leq s,t\leq\alpha$.

$$
S_1 \leq
\oh^{3n}
\sum_{u=\alpha, u+v=n}^{n-\alpha}{n\choose u}\sum_{s=\alpha}^u{u\choose s}\sum_{t=0}^v{v\choose t}\left(\oh^u+\oh^v-\oh^{t+u}-\oh^{s+v}+\oh^{s+t}\right)^{m-3} 
$$

$$
\leq\oh^{3n}\sum_{u=0, u+v=n}^n{n\choose u}\sum_{s=0}^u{u\choose s}\sum_{t=0}^v{v\choose t}\oh^{(\alpha-2)(m-3)} = O(\oh^{(\beta(\alpha-2)+1)n}),$$ which is $o(\oh^{2n})$ when choosing $\alpha$ large enough (e.g. $\alpha > 2(1+1/\beta)$).

Note that $t\geq\alpha\Rightarrow v\geq\alpha$.

$S_2\leq \oh^{3n}\sum_{u=\alpha}^{n-\alpha}{n\choose u}\sum_{s=1}^\alpha n^\alpha \sum_{t=\alpha}^v {v\choose t}\left(3\oh^\alpha\right)^{m-3} = o(\oh^{2n})$ for large enough $\alpha$.

$S_3\leq \oh^{3n}\sum_{u=1}^{\alpha}n^{2\alpha}\sum_{s=1}^\alpha \sum_{t=\alpha}^v {v\choose t}\left(\oh^\alpha+\oh+\oh^\alpha\right)^{m-3}=o(\oh^{2n})$ for large enough $\alpha$.

$S_4 \leq \oh^{3n}\alpha^3 n^{3\alpha} = o(\oh^{2n})$.

Consequently $P(c\nt b\nt a)- 3\oh^{2n} = o(\oh^{2n})$.\qedhere
\end{proof}

\subsubsection*{(v) $d\nt b\nt a$}

\begin{lemma}
\label{dba}
$P(d\nt b\nt a) \sim \oh^{m+n-2}+\oh^{2n}$.
\end{lemma}
\begin{proof}
Let $X^\prime=X-\{a,b\}$, $Y^\prime=Y-\{d\}$.

As before, we have the simple lower bound:
$P(d\nt b\nt a) \geq P(\{O_d=O_b=\emptyset\}\cup\{O_d=I_a=\emptyset\}\cup\{I_b=I_a=\emptyset\}) \geq \oh^{m+n-2}+\oh^{2n}-\oh^{m+2n-3}$.

We bound the probability from above by the probability of there being no path from $d$ to $b$ or from $b$ to $a$ of length at most $3$ or $4$ respectively.
The edges $\{a,d\}$ and $\{b,d\}$ are both directed towards $d$.
Condition on $O_d = T$, $I_b = U$, $I_a=S$, $O_b = V$.
No edge is directed from $T$ to $U$, and $S\subseteq U$.
These conditions imply that no $x\in S$ is a $VU$-witness.
In addition we forbid any $x\in X\p-T$ to be a $VS$-witness.
The events '$x$ is a $VS$-witness' are independent for $x\in X\p-T$ and independent of the other necessary events just stated.
We obtain

$$P(d\nt b\nt a)$$

$$= \sum_{S,T,U,V}P(d\nt b\nt a|O_d=T,I_b=U,O_b=V, I_a = S)P(O_d=S,I_b=T,I_a=U, I_a=S)$$

$$\leq \oh^{2n+m-2}\sum_{t=0}^{m-2}{m-2\choose t}\sum_{u=0}^{n-1}{n-1\choose u}\sum_{s=0}^u{u\choose s}\oh^{st}\left(\oh^s+\oh^v-\oh^{s+v}\right)^{m-2-t}$$
$$= \oh^{2n+m-2}\sum_{u=0}^{n-1}\sum_{s=0}^u {n-1\choose u}{u\choose s} \left(\oh^u+\oh^v+\oh^s-\oh^{s+v}\right)^{m-2}.$$

For $s = 0$ we obtain the following sum:
$\oh^{2n+m-2}\sum_{u=0}^{n-1}{n-1\choose u}\left(\oh^u+1\right)^{m-2}=
\oh^{2n+m-2}\left(2^{m-2}+(n-1)^\alpha\alpha\left(\frac{3}{2}\right)^{m-2}+\left(1+\oh^\alpha\right)^{m-2}\right)=
\oh^{2n}+(n-1)^\alpha\alpha\oh^{2n}\left(\frac{3}{2}\right)^{m-2}+\oh^{2n+m-2}\left(1+\oh^\alpha\right)^{m-2} = \oh^{2n}+o(\oh^{m+n-2}+\oh^{2n})$.

For $v=0=n-1-u$:
$\oh^{2n+m-2}\sum_{s=0}^{n-1}{n-1\choose s}\left(1+\oh^{n-1}\right)^{m-2} = 
\oh^{m+n-2}\left(1+\oh^{n-1}\right)^{m-2} = \Theta(\oh^{m+n-1})$, as
$\lim_{n\to\infty} \left(1+\oh^{n-1}\right)^{m-2}=1$.

Split the remaining sum into the following cases: 
$S_1: 1\leq u\leq \alpha$,
$S_2: 1\leq s \leq \alpha \leq u,v$,
$S_3: \alpha \leq s,u,v$,
$S_4: 1\leq s,v \leq \alpha$,
$S_5: 1\leq v\leq \alpha \leq s$.

Clearly, $S_1 = o(\oh^{m+2n-2} + \oh^{2n})$.

$S_2 \leq 
n^\alpha\oh^{m+2n-2}\sum_{u=\alpha}^{n-\alpha}{n-1\choose u}\left(2\oh^\alpha+\oh^s\right)^{m-2} = o(\oh^{m+2n-2}+\oh^{2n})$, for $\alpha > 21$, which is easily seen by considering $\beta \leq 1$ and $\beta > 1$ separately.

$S_3 \leq \oh^{m+2n-2}\sum_{u=\alpha}^{n-\alpha}{n-1\choose u}\sum_{s=\alpha}^u {u\choose s} \left(3 \cdot \oh^\alpha\right)^{m-2} \leq 
\oh^{m+2n-2}\left(\frac{3}{2^\alpha}\right)^{m-2}\cdot 3^{n-1} = o(\oh^{m+n-2}+\oh^{2n})$.

$S_4 = o(\oh^{m+2n-2} + \oh^{2n})$, since $S_4$ is the sum of a constant ($\alpha^2$) number of $o(\oh^{m+2n-2}+\oh^{2n})$ terms.

$S_5 \leq \alpha n^\alpha \oh^{m+2n-2} \sum_{s=\alpha}^{n-1} {n-1\choose s}\left(\oh^\alpha+\oh^v+\oh^\alpha\right)^{m-2} = \alpha n^\alpha \oh^{m+n}\left(\oh+\oh^{\alpha-1}\right)^{m-2} = o(\oh^{m+2n-2} + \oh^{2n})$ for $\alpha > 2$.

Finally, 
$$P(d\nt b\nt a) \sim \oh^{m+n-2}+\oh^{2n}.$$

\end{proof}

We now show how to use the lemmas above to prove Theorem 1. For example, suppose $\beta < 1$, $x,y\in X$, and $z\in Y$.
Then
$$RC_{m,n} = 1 - \frac{P(x\nt y)P(y\nt z)}{P(x\nt y\nt z)} =$$
$$1 - \frac{P(a\nt b)P(b\nt d)}{P(a\nt b \nt d)} = 
1 - \frac{P(b\nt a)P(d\nt a)}{P(d\nt b\nt a)} = (\textrm{by lemmas \ref{dba}, \ref{ba} and \ref{da}})
$$

$$1 - \frac{2\oh^n \left(\oh^m +\oh^n\right)}{4 \oh^{m+n} + \oh^{2n}}.$$

Now, since $\beta < 1$ we get
$$
\lim_{n\to\infty} RC_{m,n} = \lim_{n\to\infty} 1 - \frac{2(1-\oh^{n-m})}{4+\oh^{n-m}} = \frac{1}{2}.
$$
The other entries in the table in the statement of the theorem can be found similarly.

\section{Further questions}

I would like to mention two questions:

Question 1:

How do the results above change if an edge $e = \{x,y\}$, where $x\in X$ and $y\in Y$, rather than being oriented either way with equal probability, is directed from $X$ to $Y$ with some fixed probability $p$?

Question 2:

From Theorem 1, it seems plausible that $\{x\nt y\}$ and $\{y\nt z\}$ should be negatively correlated in any complete bipartite graph (at least in any large enough graph) when $x,y,z$ belong to the same color class.
This seems not to be the case, however; computer calculations show that if $n$ is much larger than $m$ (on the scale $n = 2^m$), then the events mentioned seem to be positively correlated even for fairly large $n$.
Is this true in general?
Can the cases with positive correlation be completely identified in terms of $m$ and $n$?

The calculations mentioned above made use of the following recursions:
$$
f_X(m,n,k)=\sum_{l=0}^n{n\choose l}\frac{(2^k-1)^l}{2^{nk}}f_Y(m-k,n,l),
$$

$$
f_Y(m,n,l) = \sum_{k = 0}^{m - 1}{m - 1\choose l}\frac{(2^k - 1)^l}{2^{mk}}f_X(m,n - l,k),
$$

$$
g_X(m,n,k) = \sum_{l=0}^n{n\choose l}\frac{(2^k-1)^l}{2^{nk}}g_Y(m-k,n,l),
$$

$$
g_Y(m,n,l)= \sum_{k=0}^{m-2}{m-2\choose k}\frac{(2^l-1)^k}{2^{lm}}g_X(m,n-l,k),
$$

$$
h_X(m,n,k) = \sum_{l=0}^{n-1}{n-1\choose l}\frac{(2^k-1)^l}{2^{nk}}h_Y(m-k,n,l),
$$

and

$$
h_Y(m,n,l) = \sum_{k=0}^{m-1}{m-1\choose k}\frac{(2^l-1)^k}{2^{mk}}h_X(m,n-l,k).
$$

where $f_X$, $f_Y$, $g_X$, $g_Y$, $h_X$ and $h_Y$ are defined as follows.

Let $P_{m,n}$ denote the probability measure associated with a uniformly chosen orientation of $K_{m,n}$, where the class $X$ has size $m$ and the class $Y$ has size $n$. For a subset $K$ of vertices of $K_{m,n}$ and a vertex $x$ in $K_{m,n}$, let $\{K \nt x\} = \bigcap_{k\in K} \{k \nt x\}$. Let $a,b,c\in X$,$d\in Y$ and $K$ be any subset of $X$, not including $a$ or $b$, of size $k$ and $L$ be any subset of $Y$, not including $d$, of size $l$.

Then $f_X(m,n,k) = P_{m,n}(K \nt a)$, $f_Y(m,n,l) = P_{m,n}(L\nt a)$, $g_X(m,n,k) = P_{m,n}(K \nt b$ and $b \nt a)$, $g_Y(m,n,l) = P_{m,n}(L\nt b$ and $b \nt a)$, $h_X(m,n,k)= P_{m,n}(K\nt d$ and $d \nt a)$, and $h_Y(m,n,l) = P_{m,n}(L \nt b$ and $b \nt a)$.

In addition, we have the following base cases for the formulas above:
$g_X(m,n,0) = g_Y(m,n,0) = f_X(m,n,1)$, $h_X(m,n,0) = h_Y(m,n,0) = f_Y(m,n,1)$, and $f_X(m, n, 0) = f_Y(m, n, 0) = 1$.

These functions are related to the quantities estimated in the lemmas above by
$f_X(m,n,1) = P_{m,n}(a\nt b)$, $f_Y(m,n,1) = P_{m,n}(a\nt d)$, $g_X(m,n,1) = P_{m,n}(c\nt b\nt a)$, $g_Y(m,n,1) = P(d\nt b\nt a)$, and $h_X(m,n,1) = P(c\nt d \nt a)$.

\section*{Acknowledgement}
I thank Svante Linusson for suggesting the problem and for helpful suggestions on earlier versions of the manuscript.

Department of Mathematics, KTH - Royal Institute of Technology, SE-100 44 Stockholm, Sweden.

E-Mail: {\tt eaas@math.kth.se.}

\end{document}